\documentclass[11pt]{amsart}

\title[On the gcd of $n$ and $F_n$, II]{On the greatest common divisor of $n$ and\\ the $n$th Fibonacci number, II}

\subjclass[2010]{Primary: 11B39, Secondary: 11A05, 11N25}

\keywords{Fibonacci numbers; greatest common divisor; rank of appearance; upper bound}

\author[A.~Jha]{Abhishek Jha}
\address{\parbox{\linewidth}{
Indraprastha Institute of Information Technology,\\
Okhla Industrial Estate, Phase-3, New Delhi, India}}
\email{abhishek20553@iiitd.ac.in}

\author[C.~Sanna]{Carlo Sanna$^\dagger$}
\address{\parbox{\linewidth}{
Department of Mathematical Sciences, Politecnico di Torino\\
Corso Duca degli Abruzzi 24, 10129 Torino, Italy\\[-8pt]}}
\email{carlo.sanna.dev@gmail.com}
\thanks{$\dagger$ C.~Sanna is a member of GNSAGA of INdAM and of
CrypTO, the group of Cryptography and Number~Theory of Politecnico di
Torino.}

\usepackage{amsmath}
\usepackage{amssymb}
\usepackage{amsthm}
\usepackage{geometry}
\usepackage{color}
\usepackage{hyperref}
\usepackage{enumitem}
\usepackage{graphicx}
\hypersetup{colorlinks=true}

\newtheorem{theorem}{Theorem}[section]

\newtheorem{lemma}[theorem]{Lemma}
\theoremstyle{remark}
\newtheorem{remark}{Remark}[section]

\DeclareMathOperator{\lcm}{lcm}
\DeclareMathOperator{\Li}{Li}

\begin{document}

\begin{abstract}
Let $\mathcal{A}$ be the set of all integers of the form $\gcd(n, F_n)$, where $n$ is a positive integer and $F_n$ denotes the $n$th Fibonacci number. Leonetti and Sanna proved that $\mathcal{A}$ has natural density equal to zero, and asked for a more precise upper bound. We prove that
\begin{equation*}
\#\big(\mathcal{A} \cap [1, x]\big) \ll \frac{x \log \log \log x}{\log \log x}
\end{equation*}
for all sufficiently large $x$.
\end{abstract}

\maketitle

\section{Introduction}

Let $(u_n)$ be a nondegenerate linear recurrence with integral values.
Arithmetic relations between $n$ and $u_n$ have been studied by several authors.
For example, the set of positive integers such that $n$ divides $u_n$ has been studied by Alba~Gonz\'{a}lez, Luca, Pomerance, and Shparlinski~\cite{MR2928495}, assuming that the characteristic polynomial of $(u_n)$ is separable, and by Andr\'e-Jeannin~\cite{MR1131414}, Luca and Tron~\cite{MR3409327}, Sanna~\cite{MR3606950}, and Somer~\cite{MR1271392}, when $(u_n)$ is a Lucas sequence.
Furthermore, Sanna~\cite{MR3935356} showed that the set of natural numbers $n$ such that $\gcd(n, u_n) = 1$ has a natural density (see~\cite{MR3896876} for a generalization).
Mastrostefano and Sanna~\cite{MR3983305,MR3825473} studied the moments of $\log\!\big(\!\gcd(n, u_n)\big)$ and $\gcd(n, u_n)$ when $(u_n)$ is a Lucas sequence, and Jha and Nath~\cite{JhaNath22} performed a similar study over shifted primes.
(See also the survey of Tron~\cite{MR4159096} on greatest common divisors of terms of linear recurrences.)

Let $(F_n)$ be the linear recurrence of Fibonacci numbers, which is defined by $F_1 = F_2 = 1$ and $F_{n + 2} = F_{n + 1} + F_n$ for every positive integer $n$.
Sanna and Tron~\cite{MR3800663} proved that, for each positive integer $k$, the set of positive integers $n$ such that $\gcd(n, F_n) = k$ has a natural density, which is given by an infinite series. 
Kim~\cite{MR4017936} and Jha~\cite{Jha21} obtained formally analogous results in cases of elliptic divisibility sequences and orbits of polynomial maps, respectively.
Let $\mathcal{A}$ be the set of numbers of the form $\gcd(n, F_n)$, for some positive integer $n$.
Leonetti and Sanna~\cite{MR3859754} provided an effective method to enumerate the elements of $\mathcal{A}$ in increasing order.
In particular, the first elements of $\mathcal{A}$ are
\begin{equation*}
1, \quad 2, \quad 5, \quad 7, \quad 10, \quad 12, \quad 13, \quad 17, \quad 24, \quad 25, \quad 26, \quad 29, \quad 34, \quad 35, \quad 36, \quad \dots
\end{equation*}
see~\cite[A285058]{OEIS} for more terms.
Then they proved that 
\begin{equation}\label{equ:lbound}
\#\mathcal{A}(x) \gg \frac{x}{\log x}
\end{equation}
for all $x \geq 2$. Their approach relied on a result of Cubre and Rouse~\cite{MR3251719}, which in turn follows from Galois theory and the Chebotarev density theorem. Later, Jha and Sanna~\cite[Proposition 1.4]{JhaSanna22} obtained an elementary proof as an application of related arithmetic problem over shifted primes. Leonetti and Sanna~\cite{MR3859754} also gave the upper bound $\#\mathcal{A}(x) = o(x)$ as $x \to +\infty$; and asked for a more precise estimate.
We prove the following upper bound on $\#\mathcal{A}(x)$.

\begin{theorem}\label{thm:ubound}
We have
\begin{equation*}
\#\mathcal{A}(x) \ll \frac{x \log \log \log x}{\log \log x}
\end{equation*}
for all sufficiently large $x$.
\end{theorem}

In light of the gap between the upper bound of Theorem~\ref{thm:ubound} and the lower bound~\eqref{equ:lbound} it is natural to wonder which is the true order of $\#\mathcal{A}(x)$.
By performing some numerical experiments (see Section~\ref{sec:numerical} later), we found that $\#\mathcal{A}(x)$ appears to be asymptotic to $x / (\log x)^c$, as \mbox{$x \to +\infty$}, for some constant $c \approx 0.63$, see Figure~\ref{fig:asymp}.
Of course, these kind of experiments has to be taken with a grain of salt, since they cannot reveal slow-growing  factors like $\log \log x$.

\begin{figure}[ht]
    \centering
    \includegraphics[width=0.9\textwidth]{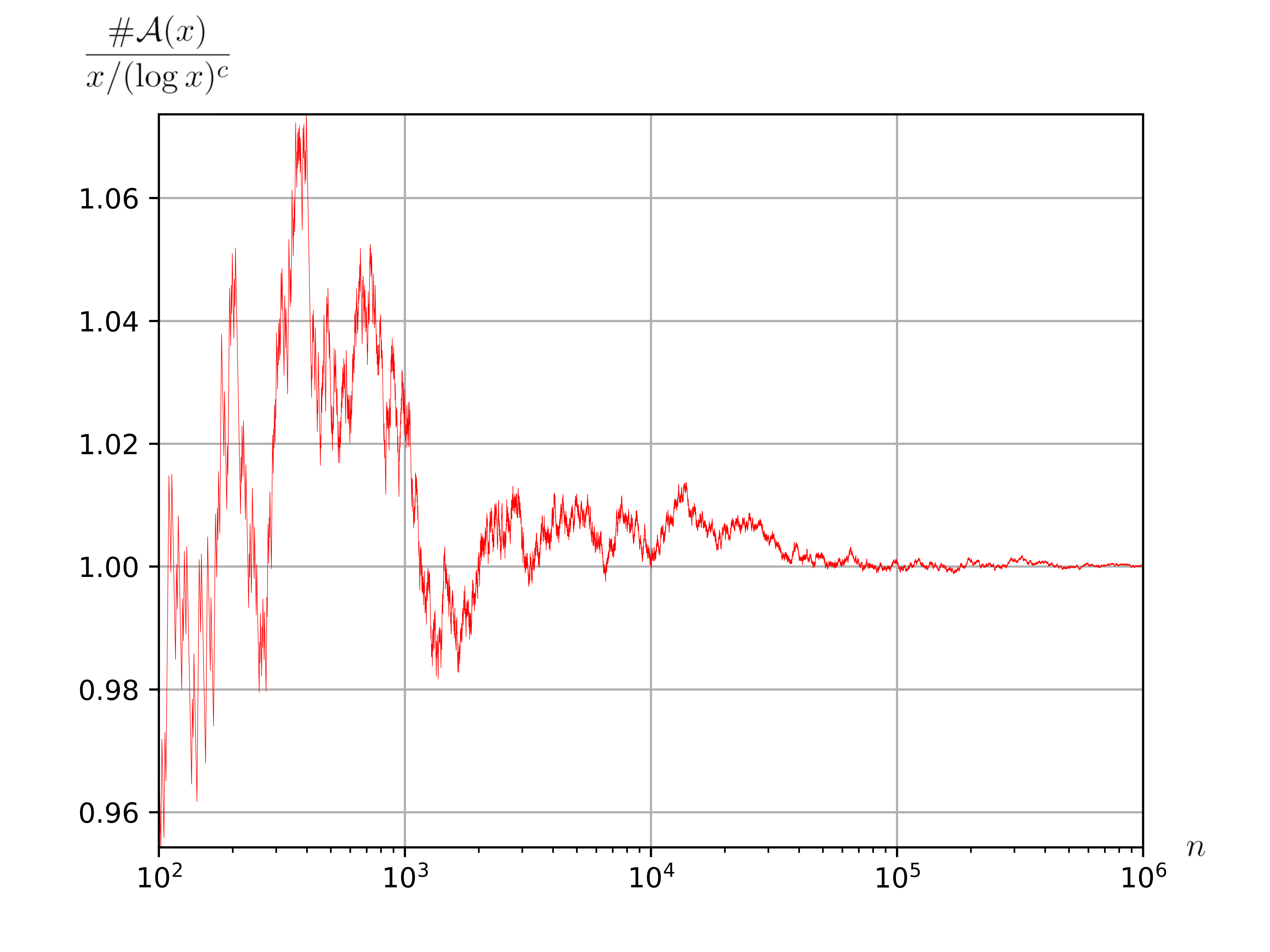}
    \caption{A plot of $\#\mathcal{A}(x) / (x / (\log x)^c)$ for $x$ up to $10^6$.}
    \label{fig:asymp}
\end{figure}

\subsection*{Notation}

For every set of positive integers $\mathcal{S}$ and for every $x > 0$, we define $\mathcal{S}(x) := \mathcal{S} \cap [1, x]$.
We employ the Landau--Bachmann ``Big Oh'' and ``little oh'' notation $O$ and $o$, as well as the associated Vinogradov symbols $\ll$ and $\gg$.
In particular, all of the implied constants are intended to be absolute.
We let $\Li(x) := \int_2^x (\log t)^{-1}\,\mathrm{d}t$ denote the integral logarithm.

\section{Preliminaries}

For each positive integer $n$, let $z(n)$ be the \emph{rank of appearance} of $n$, that is, $z(n)$ is the smallest positive integer $k$ such that $n$ divides $F_k$.
It is well known that $z(n)$ exists.
Moreover, put $\ell(n) := \lcm\!\big(n, z(n)\big)$ and $g(n) := \gcd\!\big(n, F_{n}\big)$.
The next lemma collects some elementary properties of $z$, $\ell$, and $g$.

\begin{lemma}\label{lem:basic}
For all positive integer $m,n$ and all prime numbers $p$, we have:
\begin{enumerate}[label={\rm (\roman{*})},itemsep=0.5em]
\item \label{item1b} $z(m) \mid z(n)$ whenever $m \mid n$.
\item \label{item2b} $n \mid g(m)$ if and only if $\ell(n) \mid m$.
\item \label{item3b} $n \in \mathcal{A}$ if and only if $n = g\big(\ell(n)\big)$.
\item \label{item4b} $m \mid n$ whenever $\ell(m) \mid \ell(n)$ and $n \in \mathcal{A}$.
\item \label{item5b} $z(p)\mid p - (p/5)$ where $(p/5)$ is a Legendre symbol.
\item \label{item6b} $z(p^n)=p^{\max(n-e(p)\,,\,0)}\,z(p)$, where $e(p) := \nu_p(F_{z(p)})\ge 1$ and $\nu_p$ is the usual p-adic valuation. 
\item \label{item7b} $\ell(p^n)=p^n z(p)$ if $p\neq 5$, and $\ell(5^n)=5^n$.
\end{enumerate}
\end{lemma}
\begin{proof}
For~\ref{item1b}, \ref{item2b}. and~\ref{item3b}, see~\cite[Lemma~2.1 and~2.2]{MR3859754}.
Fact~\ref{item4b} follows easily from~\ref{item2b} and~\ref{item3b}.
Facts~\ref{item6b} and ~\ref{item5b} are well known (cf.~\cite[Lemma~1]{MR3409327}).
Fact~\ref{item7b} follows quickly from~\ref{item6b} and~\ref{item5b}.
\end{proof}

For each positive integer $d$, let $\mathcal{P}_d$ be the set of prime numbers $p$ such that $d$ divides $z(p)$.
Cubre and Rouse~\cite{MR3251719} proved that $\#\mathcal{P}_d(x) \sim \delta(d) \Li(x)$, as $x \to +\infty$, where
\begin{equation*}
\delta(d) := \frac1{d}\prod_{p \,\mid\, d} \left(1 - \frac1{p^2}\right)^{-1} \begin{cases} 1 & \text{ if } 10 \nmid d; \\ 5/4 & \text{ if } d \equiv 10 \!\!\!\pmod{20}; \\ 1/2 & \text{ if } 20 \mid d . \end{cases}
\end{equation*}
Sanna~\cite{San22} extended this result to Lucas sequences (under some mild restrictions) and provided also an error term.
In particular, as a consequence of~\cite[Theorem~1.1]{San22}, we have the following asymptotic formula.

\begin{lemma}\label{lem:Pdx-asymp}
There exists an absolute constant $B > 0$ such that
\begin{equation}\label{equ:Pdx-asymp}
\#\mathcal{P}_d(x) = \delta(d) \Li(x) + O\!\left(\frac{x}{(\log x)^{12/11}}\right) ,
\end{equation}
for all odd positive integers $d$ and for all $x \geq \exp(B d^{40})$.
\end{lemma}
\begin{proof}
From~\cite[Theorem~1.1]{San22} we have that there exists an absolute constant $B > 0$ such that
\begin{equation*}
\#\mathcal{P}_d(x) = \delta(d) \Li(x) + O\!\left(\frac{d}{\varphi(d)} \cdot \frac{x \, (\log \log x)^{\omega(d)}}{(\log x)^{9/8}}\right) ,
\end{equation*}
for all odd positive integers $d$ and for all $x \geq \exp(B d^{40})$, where $\varphi(d)$ and $\omega(d)$ are the Euler totient function and the number of prime factors of $d$, respectively.
Note that we can assume that $B$ (and consequently $x$) is sufficiently large.
In particular, we have that $d \leq (\log x)^{1/40}$.
Put $\varepsilon := 1/330$.
By the classic lower bound for $\varphi(d)$ (see, e.g.,~\cite[Ch.~I.5, Theorem~5.6]{MR3363366}) we have that
\begin{equation*}
\frac{d}{\varphi(d)} \ll \log \log d \ll \log \log \log x \leq (\log x)^{\varepsilon} .
\end{equation*}
Recall that $\omega(d) \leq \big(1 + o(1)\big) \log d / \log \log d$ as $d \to +\infty$ (see, e.g.,~\cite[Ch.~I.5, Theorem~5.5]{MR3363366}).
Therefore, there exists an absolute constant $C > 0$ such that if $d > C$ then
\begin{equation*}
\omega(d) \leq (1 + \varepsilon) \frac{\log d}{\log \log d} \leq \left(\frac1{40} + 2\varepsilon\right) \frac{\log \log x}{\log \log \log x} ,
\end{equation*}
and consequently $(\log \log x)^{\omega(d)} \leq (\log x)^{\frac1{40} + 2\varepsilon}$ .
Also, if $d \leq C$ then $(\log \log x)^{\omega(d)} \leq (\log x)^\varepsilon$.
The claim follows.
\end{proof}

\begin{remark}
In Lemma~\ref{lem:Pdx-asymp} the exponent $12/11$ can be replaced by $11/10 + \varepsilon$, for every $\varepsilon > 0$, assuming that $x$ is sufficiently large depending on $\varepsilon$.
\end{remark}

We also need an upper bound for $\#\mathcal{P}_d(x)$.

\begin{lemma}\label{lem:Pdx-upper}
We have
\begin{equation*}
\#\mathcal{P}_d(x) \ll \frac{x}{\varphi(d) \log (x / d)}
\end{equation*}
for all positive integers $d$ and for all $x > d$.
\end{lemma}
\begin{proof}
By Lemma~\ref{lem:basic}\ref{item5b}, we have that
\begin{align*}
\#\mathcal{P}_d(x) \leq 1 + \#\big\{p \leq x : p \equiv \pm 1 \!\!\!\pmod d \big\} \ll \frac{x}{\varphi(d) \log (x / d)} ,
\end{align*}
where we applied the Brun--Titchmarsh inequality~\cite[Ch.~I.4, Theorem~4.16]{MR3363366}.
\end{proof}

Now we give an upper bound for the sum of reciprocals of primes in $\mathcal{P}_d$.

\begin{lemma}\label{lem:reci}
We have
\begin{equation*}
\sum_{p \,\in\, \mathcal{P}_d(x)} \frac1{p} = \delta(d)\log \log x + O\!\left(\frac{\log(2d)}{\varphi(d)}\right)
\end{equation*}
for all odd positive integers $d$ and for all $x \geq 3$.
\end{lemma}
\begin{proof}
First, suppose that $x < \exp(Bd^{40})$, where $B$ is the constant of Lemma~\ref{lem:Pdx-asymp}. 
Hence, we have that
\begin{equation*}
\delta(d)\log \log x \ll \frac{\log \log x}{d} \ll \frac{\log(2d)}{d} .
\end{equation*}
Moreover, by~\cite[Theorem~1 and Remark~1]{MR447087}, we have that
\begin{equation*}
\sum_{\substack{p \,\leq\, x \\[1pt] p \,\equiv\, \pm 1 \!\!\!\!\pmod d}} \frac1{p} = \frac{2\log \log x}{\varphi(d)} + O\!\left(\frac{\log(2d)}{\varphi(d)}\right) .
\end{equation*}
This together with Lemma~\ref{lem:basic}\ref{item5b} yields that
\begin{equation*}
\sum_{p \,\in\, \mathcal{P}_d(x)} \frac1{p} \leq \frac1{d} + \sum_{\substack{p \,\leq\, x \\[1pt] p \,\equiv\, \pm 1 \!\!\!\!\pmod d}} \frac1{p} \ll \frac1{d} + \frac{\log (2d)}{\varphi(d)} .
\end{equation*}
Hence, the claim follows.
Now suppose that $x \geq \exp(Bd^{40})$.
By partial summation, we have that
\begin{equation*}
\sum_{p \,\in\, \mathcal{P}_d(x)} \frac1{p} = \frac{\#\mathcal{P}_d(x)}{x} + \int_1^x \frac{\#\mathcal{P}_d(t)}{t^2} \,\mathrm{d} t .
\end{equation*}
Obviously, $\#\mathcal{P}_d(x) / x \ll 1/d$ by the trivial inequality.
Thus it remains to bound the integral.
By Lemma~\ref{lem:basic}\ref{item5b}, we have that
\begin{equation*}
\int_1^{2d} \frac{\#\mathcal{P}_d(t)}{t^2} \,\mathrm{d} t \leq \frac1{d^{2}}\int_1^{2d-2} 5 \,\mathrm{d} t \ll \frac1{d},
\end{equation*} after noticing that $\#\mathcal{P}_d(t)>0$ only if $t\ge d-1$.
By Lemma~\ref{lem:Pdx-upper}, we have that
\begin{equation*}
\int_{2d}^{\exp(Bd^{40})} \frac{\#\mathcal{P}_d(t)}{t^2} \,\mathrm{d} t 
    \ll \int_{2d}^{\exp(Bd^{40})} \frac{\mathrm{d} t}{\varphi(d) \,t \log(t / d)} = \left[ \frac{\log\log(t/d)}{\varphi(d)}\right]_{t \,=\, 2d}^{\exp(Bd^{40})} \ll \frac{\log d}{\varphi(d)}.
\end{equation*}
By Lemma~\ref{lem:Pdx-asymp}, we have that
\begin{align*}
\int_{\exp(Bd^{40})}^x \frac{\#\mathcal{P}_d(t)}{t^2} \,\mathrm{d} t 
    &= \int_{\exp(Bd^{40})}^x \frac{\delta(d) \Li(t)}{t^2} \,\mathrm{d} t + O\!\left(\int_{\exp(Bd^{40})}^x \frac{\mathrm{d}t}{t (\log t)^{12/11}} \right) \\
    &= \delta(d) \left[\log \log t - \frac{\Li(t)}{t} \right]_{t \,=\, \exp(Bd^{40})}^x + O\!\left( \frac1{d^{40/11}} \right) \\
    &= \delta(d) \left(\log \log x + O(\log d) \right) + O\!\left( \frac1{d^{40/11}} \right) \\
    &= \delta(d) \log \log x + O\!\left(\frac{\log d}{d} \right) .
\end{align*}
Putting these together, the claim follows.
\end{proof}

The following sieve result is a special case of \cite[Theorem~7.2]{MR0424730} (cf.~\cite[Lemma~2.2]{MR4356164}).

\begin{lemma}\label{lem:sieve} 
We have
\begin{equation*}
 \#\{n\le x : p\mid n \Rightarrow p \notin \mathcal{P}\} \ll x \prod_{p \,\in\, \mathcal{P}(x)} \left(1 - \frac{1}{p}\right) ,
\end{equation*}
for all $x \ge 2$ and for all sets of prime numbers $\mathcal{P}$.
\end{lemma}

\section{Proof of Theorem~\ref{thm:ubound}}

Suppose that $x > 0$ is sufficiently large, and put
\begin{equation*}
    k := \left\lfloor \frac1{\log 5} \log\!\left(\frac{25}{24\log 5} \cdot \frac{\log \log x}{\log \log \log x}\right)\right\rfloor 
\end{equation*}
and $d := 5^k$.
Note that $\delta(d) = 5^{-k} \cdot 25 / 24$.
Hence, we get that
\begin{equation*}
    \log \!\left(\frac{\log d}{\delta(d)}\right) = k \log 5 + \log k + \log\!\left(\frac{24 \log 5}{25}\right) \leq \log \log \log x .
\end{equation*}
Therefore, we have that $(\log d) / \delta(d) \leq \log \log x$ and
\begin{equation}\label{equ:optimize}
    (\log x)^{\delta(d)} \geq d \gg \frac{\log \log x}{\log \log \log x} .
\end{equation}
We split $\mathcal{A}$ into two subsets: $\mathcal{A}_1$ is the subset of $\mathcal{A}$ consisting of integers without prime factors in $\mathcal{P}_d$, and $\mathcal{A}_2 := \mathcal{A} \setminus \mathcal{A}_1$.

First, we give an upper bound on $\#\mathcal{A}_1(x)$.
By Lemma~\ref{lem:sieve} and Lemma~\ref{lem:reci}, we get that 
\begin{equation}\label{equ:A1-bound}
    \#\mathcal{A}_1(x) \ll x \prod_{p \,\in\, \mathcal{P}_d(x)} \left(1-\frac{1}{p}\right) 
    \ll x \exp\!\left(-\!\!\sum_{p \,\in\, \mathcal{P}_d(x)} \frac1{p}\right)
    \ll \frac{x}{(\log x)^{\delta(d)}} ,
\end{equation}
where we also used the inequality $1 - x \leq \exp(-x)$, which holds for $x \geq 0$.

Now we give an upper bound on $\#\mathcal{A}_2(x)$.
If $n \in \mathcal{A}_2$ then $n$ has a prime factor $p \in \mathcal{P}_d$.
Hence, we have that $p \mid n$ and $d \mid z(p)$.
Thus, by Lemma~\ref{lem:basic}\ref{item1b}, we get that $z(p) \mid z(n)$ and so $d \mid \ell(n)$.
Recalling that $d = 5^k$, by Lemma~\ref{lem:basic}\ref{item7b} we have that $\ell(d) = d$.
Hence, we get that $\ell(d) \mid \ell(n)$ and, by Lemma~\ref{lem:basic}\ref{item4b}, it follows that $d \mid n$.
Thus all the elements of $\mathcal{A}_2$ are multiples of $d$.
Consequently, we have that
\begin{equation}\label{equ:A2-bound}
    \#\mathcal{A}_2(x) \leq \frac{x}{d} .
\end{equation}
Therefore, putting together~\eqref{equ:A1-bound} and~\eqref{equ:A2-bound}, and using~\eqref{equ:optimize}, we obtain that
\begin{equation*}
    \#\mathcal{A}(x) = \#\mathcal{A}_1(x) + \#\mathcal{A}_2(x) \ll \frac{x}{(\log x)^{\delta(d)}} + \frac{x}{d} \ll \frac{x \log \log \log x}{\log \log x} ,
\end{equation*}
as desired.
The proof is complete.

\section{Numerical computations}\label{sec:numerical}

We computed the elements of $\mathcal{A} \cap [1, 10^6]$ by using a program written in \textsf{C} that employs Lemma~\ref{lem:basic}\ref{item3b}. 
Note that computing $g\big(\ell(n)\big)$ directly as $\gcd\!\big(\ell(n), F_{\ell(n)}\big)$ would be prohibitive, in light of the exponential grown of Fibonacci numbers.
Instead, we used the fact that 
\begin{equation*}
g\big(\ell(n)\big) = \gcd\!\big(\ell(n), F_{\ell(n)} \bmod \ell(n)\big),
\end{equation*}
and we computed Fibonacci numbers modulo an integer by efficient matrix exponentiation.

\bibliographystyle{amsplain-no-bysame}
\bibliography{main}

\end{document}